\documentclass[12pt]{amsart}
\usepackage{amscd,amsmath,amsthm,amssymb,graphics}
\usepackage{pstcol,pst-plot,pst-3d}
\usepackage{lmodern,pst-node}
\usepackage[dvips]{graphicx}
\usepackage{multicol}
\usepackage{epic,eepic}
\usepackage{amsfonts,amssymb,amscd,amsmath,enumitem,verbatim}
\psset{unit=0.7cm,linewidth=0.8pt,arrowsize=2.5pt 4}

\newpsstyle{fatline}{linewidth=1.5pt}
\newpsstyle{fyp}{fillstyle=solid,fillcolor=verylight}
\definecolor{verylight}{gray}{0.97}
\definecolor{light}{gray}{0.9}
\definecolor{medium}{gray}{0.85}
\definecolor{dark}{gray}{0.6}

\usepackage{lineno}

\def\frk{\frak}               

\def\mm{{\frk m}}

\def\Phi{{\frk n}}
\def\Phi{{\frk N}}
\def\MS{{\mathcal S}}

\def\MC{{\mathcal C}}

\def\zb{{\bold z}}

\def\opn#1#2{\def#1{\operatorname{#2}}} 
\opn\chara{char} \opn\length{\ell} \opn\pd{pd} \opn\rk{rk}
\opn\projdim{proj\,dim} \opn\injdim{inj\,dim} \opn\rank{rank}
\opn\depth{depth} \opn\grade{grade} \opn\height{height}
\opn\embdim{emb\,dim} \opn\codim{codim}

\opn\Tr{Tr} \opn\bigrank{big\,rank}
\opn\superheight{superheight}\opn\lcm{lcm}
\opn\trdeg{tr\,deg}
\opn\reg{reg} \opn\lreg{lreg} \opn\ini{in} \opn\lpd{lpd}
\opn\size{size}\opn\bigsize{bigsize}
\opn\cosize{cosize}\opn\bigcosize{bigcosize}
\opn\sdepth{sdepth}\opn\sreg{sreg}
\opn\link{link}\opn\fdepth{fdepth}
\opn\div{div} \opn\Div{Div} \opn\cl{cl} \opn\Cl{Cl}
\let\epsilon\varepsilon
\let\phi=\varphi
\let\kappa=\varkappa
\opn\Spec{Spec} \opn\Supp{Supp} \opn\supp{supp} \opn\Sing{Sing}
\opn\Ass{Ass} \opn\Min{Min}\opn\Mon{Mon} \opn\dstab{dstab} \opn\astab{astab}
\opn\Syz{Syz}
\opn\Ann{Ann} \opn\Rad{Rad} \opn\Soc{Soc}
\opn\Im{Im} \opn\Ker{Ker} \opn\Coker{Coker} \opn\Am{Am}
\opn\Hom{Hom} \opn\Tor{Tor} \opn\Ext{Ext} \opn\End{End}
\opn\Aut{Aut} \opn\id{id}

\opn\nat{nat}
\opn\pff{pf}
\opn\Pf{Pf} \opn\GL{GL} \opn\SL{SL} \opn\mod{mod} \opn\ord{ord}
\opn\Gin{Gin} \opn\Hilb{Hilb}\opn\sort{sort}
\opn\initial{init}
\opn\ende{end}
\opn\height{height}
\opn\type{type}
\opn\set{set}
\opn\indeg{indeg}
\opn\aff{aff} \opn\con{conv} \opn\relint{relint} \opn\st{st}
\opn\lk{lk} \opn\cn{cn} \opn\core{core} \opn\vol{vol}
\opn\link{link} \opn\star{star}\opn\lex{lex}
\opn\gr{gr}

\def\pot#1#2{#1[\kern-0.28ex[#2]\kern-0.28ex]}

\opn\dirlim{\underrightarrow{\lim}}
\opn\inivlim{\underleftarrow{\lim}}
\let\to=\rightarrow

\def\Implies{\ifmmode\Longrightarrow \else
        \unskip${}\Longrightarrow{}$\ignorespaces\fi}
\def\implies{\ifmmode\Rightarrow \else
        \unskip${}\Rightarrow{}$\ignorespaces\fi}
\def\iff{\ifmmode\Longleftrightarrow \else
        \unskip${}\Longleftrightarrow{}$\ignorespaces\fi}

\let\:=\colon
 \theoremstyle{plain}
\newtheorem{Theorem}{Theorem}[section]
 \newtheorem{Lemma}[Theorem]{Lemma}
 \newtheorem{Corollary}[Theorem]{Corollary}
 \newtheorem{Proposition}[Theorem]{Proposition}

 \theoremstyle{definition}
 \newtheorem{Definition}[Theorem]{Definition}

\let\epsilon\varepsilon
\let\kappa=\varkappa
\opn\dis{dis}
\def\pnt{{\raise0.5mm\hbox{\large\bf.}}}

\opn\Lex{Lex}

\begin{document}
\title{Licci binomial edge ideals}
\author {Viviana Ene, Giancarlo Rinaldo, Naoki Terai}

\address{Viviana Ene, Faculty of Mathematics and Computer Science, Ovidius University, Bd. Mamaia 124, 900527 Constanta, Romania}  \email{vivian@univ-ovidius.ro}

\address{Giancarlo Rinaldo, Department of Mathematics,  University of Trento, Via Sommarive,
14  38123 Povo (Trento), Italy}  \email{giancarlo.rinaldo@unitn.it}

\address{Naoki Terai, Faculty of Education, Saga University, Saga 840-8502, Japan}  \email{terai@cc.saga-u.ac.jp}

\thanks{The second author was supported by GNSAGA of INdAM (Italy). The third author was supported by the JSPS Grant-in Aid for Scientific Research (C)  18K03244. }

\begin{abstract}
We give a complete characterization of graphs whose binomial edge ideal is licci. An important tool is a new general upper bound for the regularity of binomial edge ideals. 
\end{abstract}


\maketitle
\section*{Introduction}
\label{Intro}

Binomial edge ideals associated to simple graphs have been intensively studied in the last decade. Their algebraic and homological properties are intimately related to the combinatorics of the underlying graph. A lot of effort has been dedicated to 
study the Cohen-Macaulay property of these ideals. As in the case of classical edge ideals, an exhaustive classification of graphs whose binomial edge ideals are Cohen-Macaulay seems to be a hopeless task. There are successful attempts to characterize graphs with specific  properties which have Cohen-Macaulay binomial edge ideals. For example, the Cohen-Macaulay   property of binomial edge ideals is known for block  graphs which include the  trees
\cite{EHH} and  for bipartite graphs \cite{BMS}. We refer also to the papers \cite{KiMa2, RaRi, Rin, Rin2} for other classes of Cohen-Macaulay binomial edge ideals. 

Let $G$ be a simple graph (that is, undirected, with no loops, and no multiple edges) on the vertex set $[n]:=\{1,2,\ldots,n\}$ and $S=K[x_1,\ldots,x_n,y_1,\ldots y_n]$ the polynomial ring in $2n$ variables. The binomial edge ideal $J_G\subset S$ of $G$ is generated by all the binomials of the form $f_{ij}=x_iy_j-x_jy_i$ where 
$\{i,j\}$ is an edge of $G.$ In other words, $J_G$ is generated by  the $2$-minors of the generic matrix $X=\left(
\begin{array}{cccc}
	x_1 & x_2 & \ldots & x_n\\
	y_1 & y_2 & \ldots & y_n
\end{array}\right)
$ which correspond to the edges of $G.$

In this paper, we study  binomial edge ideals which are in the linkage class of a  complete intersection. We call such ideals licci, in brief.  Besides the Cohen-Macaulay property, they satisfy some extra conditions which make possible a full characterization of graphs whose binomial edge ideals are licci. Linkage theory has a rich history in commutative algebra and algebraic geometry.  Peskine and Szpiro \cite{PeSz} in 1974 reduced general linkage to  questions on ideals over commutative algebras and after then, a lot of work has been done to develop this theory in commutative algebra and algebraic geometry. If $I,J$ are proper ideals in a local regular ring $R,$ they are called \emph{directly linked} and we write $I\sim J$ if there exists a regular sequence 
$\zb =z_1,\ldots,z_g$ in $I\cap J$ such that $J=(\zb):I$ and $I=(\zb):J.$ One says that  $I$ and $J$ belong to the same \emph{linkage class} if there exists a sequence of direct links 
\[
I=I_0 \sim I_1\sim \cdots \sim I_m=J.
\] If $J$ is a complete intersection ideal, then $I$ is said to be licci. The ideals in the same linkage class share several properties. For example, if $I$ and $J$ are linked, then $I$ is Cohen-Macaulay if and only if $J$ is Cohen-Macaulay. In particular, it follows that 
a licci ideal is Cohen-Macaulay. 

The following natural question arises. May we give a full characterization of the graphs $G$ with the property that the associated binomial edge ideal is licci?

In this paper, we give a complete answer to this question. In \cite{HuUl} a necessary condition for a Cohen-Macaulay  homogeneous ideal in a polynomial ring to be licci is given. In the case of binomial edge ideals, this condition implies that if $(J_G)_\mm\subset S_\mm$ (here $\mm$ is the maximal graded ideal of the ring $S$) is licci, then 
$\reg(S/J_G)\geq n-2.$ This condition turns to be also sufficient for Cohen-Macaulay binomial edge ideals as we are going to show in this paper.  

The regularity of binomial edge ideals have been intensively studied in the last years. In \cite{MaMu} it was proved that the regularity of $S/J_G$ is upper bounded by $n-1$ and it was conjectured that this upper bound is attained if and only if $G$ is a path graph. This conjecture was later proved in \cite{KiMa}. Inspired by the paper \cite{KiMa},  we prove a new upper bound for $\reg(S/J_G)$ which is stronger than $n-1$ and it plays an essential role in the characterization of the graphs $G$ whose binomial edge ideal is licci. 

The structure of the paper is as follows. In Section~\ref{Prelim},  we recall the basic results on licci and binomial edge ideals needed in the next sections. In Section~\ref{Bound}, we prove that if $G$ is a connected graph, then $\reg(S/J_G)\leq n-\dim \Delta(G),$ where 
$\Delta(G)$ is the clique complex of $G$ (Theorem~\ref{NaokiThm}). We believe that this new general upper bound for the regularity of binomial edge ideals will inspire new results on their resolution. In brief, in Theorem~\ref{NaokiThm}, we prove that for every clique 
$W\subset [n]$ of the connected graph $G,$ we have $\reg(S/J_G)\leq n-|W|+1.$ The proof is based on a double induction. First we make induction on $n-|W|$ and, secondly, on a combinatorial invariant of $G.$ 

The characterization of graphs whose binomial edge ideal is licci is given in Section~\ref{Liccisec}. In Theorem~\ref{thm:licci} we show that, for a connected graph $G$ on $n$ vertices, the following statements are equivalent:
\begin{itemize}
	\item [(i)]  $(J_G)_{\mm}\subset S_\mm$ is licci.
	\item [(ii)] $J_G$ is  Cohen-Macaulay and $n-2\leq \reg(S/J_G)\leq n-1.$
	\item [(iii)] $G$ is a path graph or it is a triangle with possibly some paths attached to some of its vertices. 
\end{itemize}
The most technical part in the proof is to show that there is no indecomposable graph $G$ with $n\geq 4$ vertices with 
$\reg(S/J_G)=n-2$ and $J_G$  Cohen-Macaulay. In order to make this part easier to understand, we proved some preparatory lemmas. We can reformulate the above statement by saying that the only indecomposable graphs $G$ with $J_G$ a Cohen-Macaulay ideal and $\reg(S/J_G)=n-2$ are the path with one edge and the triangle. Next we combine this fact with Lemma~\ref{lm:decompose} which shows that for any decomposable graph $G$ with $\reg(S/J_G)=n-2,$ one of the components must be a path. In this way we derive the combinatorial characterization from
Theorem~\ref{thm:licci} (iii).

A straightforward consequence of Theorem~\ref{thm:licci} is Corollary~\ref{cor:bipartite} which says that  for a connected bipartite graph $G$, the ideal $(J_G)_{\mm}\subset S_\mm$ is licci if and only if $G$ is a path graph. The case when $G$ is a disconnected graph is treated in Proposition~\ref{pr:disconnect}.

In the last section of the paper, we show that for chordal graphs,  in  the  equivalent statements of Theorem~\ref{thm:licci}, we may replace the Cohen-Macaulay property with the unmixedness of the ideal $J_G$ (Theorem~\ref{thm:chordal}). For the proof we use a theorem of Dirac which characterizes the chordal graphs in terms of their clique complex.

\section{Preliminaries}
\label{Prelim}

We recall some notions and fundamental results needed  in the later sections.

\subsection{Licci ideals}

Let $R$ be a  regular local ring and $I,J$ proper ideals of $R.$ Then $I$ and $J$ are called \emph{directly linked} and we write $I \sim  J$ if there exists a regular sequence 
$\zb =z_1,\ldots,z_g$ in $I\cap J$ such that $J=(\zb):I$ and $I=(\zb):J.$ One says that $I$ is \emph{linked} to $J$ or that $I$ and $J$ belong to the same \emph{linkage class} if there exists a sequence of direct links 
\[
I=I_0 \sim I_1\sim \cdots \sim I_m=J.
\] If $J$ is a complete intersection ideal, that is, it is generated by a regular sequence, then $I$ is said to be in the \textbf{li}nkage \textbf{c}lass of a \textbf{c}omplete \textbf{i}ntersection 
(\textit{licci }in brief). 



Several properties are preserved in the same linkage class. For example, if $I$ is linked to  $J$, then $R/I$ is Cohen-Macaulay if and only if $R/J$ is Cohen-Macaulay \cite{PeSz}.
In particular, any licci ideal is Cohen-Macaulay. A necessary condition for a homogeneous ideal in a polynomial ring to be licci is given in \cite{HuUl}.

\begin{Theorem}\cite[Corollary 5.13]{HuUl}\label{HUineq} Let  $I$ be  a Cohen-Macaulay homogeneous ideal  in a standard graded polynomial ring $S=K[x_1,\ldots,x_n]$ with the graded maximal ideal $\mm.$ If $I_{\mm}\subset R=S_{\mm}$ is licci, then
\begin{equation}\label{HUineqeq}
\reg(S/I)\geq (\height I-1)(\indeg I-1)
\end{equation}
where $\indeg I$ is the initial degree of the ideal $I,$ that is, $\indeg I=\min\{i: I_i\neq 0\}.$
\end{Theorem}

Although, in general, inequality (\ref{HUineqeq}) is not a sufficient condition, if $I$ is the edge ideal of a graph, then $I_{\mm}\subset R=S_{\mm}$ is licci if and only if inequality (\ref{HUineqeq})
holds \cite{KTY}. We will see a similar behavior in Section~\ref{Liccisec} for binomial edge ideals.
 
\subsection{Graphs and binomial edge ideals}

Let $G$ be a simple graph on the vertex set $V(G):=[n]$ with the edge set $E(G)$ and $S=K[x_1,\ldots,x_n,y_1,\ldots y_n]$  the polynomial ring in $2n$ variables over a field $K.$
The binomial edge ideal of the graph $G$ is generated by the binomials $f_{e}:=x_iy_j-x_jy_i$ with $e=\{i,j\}\in E(G).$ In other words, $J_G$ is generated by the $2$-minors of the matrix 
$X=\left(
\begin{array}{cccc}
	x_1 & x_2 & \ldots & x_n\\
	y_1 & y_2 & \ldots & y_n
\end{array}\right)
$ which correspond to the edges of $G.$ For example, if $G$ is the complete graph $K_n$ on $n$ vertices, then $J_G$ is the ideal $I_2(X)$ generated by all the $2$-minors of $X,$ while 
if $G$ is the path graph $P_n$ on $n$ vertices with edge set $\{\{i,i+1\}:1\leq i\leq n-1\},$ then $J_G$ is the ideal of all adjacent maximal  minors of $X.$

The binomial edge ideals were introduced independently in the papers \cite{HHHKR} and \cite{Ohtani}. In the last decade,  these ideals have been studied by many authors. The interested  
reader may find a thorough introduction to this topic in the monograph \cite{HHO}. Fundamental results regarding the minimal free resolutions of binomial edge ideals are surveyed in \cite{Sara}.

In this paper, we need to recall the primary decomposition of binomial edge ideals and some fundamental results on their regularity.

The minimal primary decomposition of a binomial edge ideal is strongly related to the combinatorics of the underlying graph; see \cite{HHHKR} or \cite[Chapter 7]{HHO}. Let $\MS$ be a (possibly empty) subset of $[n]$ and let $G_{\MS}$ be the restriction of $G$ to the vertex subset 
$[n]\setminus \MS.$ Let  $G_1,\ldots,G_{c(\MS)}$ be the connected components of this restriction  and, for every $1\leq i\leq c(\MS),$ let $\tilde{G_i}$ be the complete graph on 
$V(G_i).$ Then, the ideal \[P_{\MS}(G)=(\{x_i,y_i:i\in \MS\})+J_{\tilde{G_1}}+\cdots +J_{\tilde{G}_{c(\MS)}}\] is a prime ideal in $S$ which contains $J_G$, and by  \cite[Lemma 3.1]{HHHKR} we have  
\begin{equation}\label{eq:ht}
\height(P_{\MS}(G))=n-c(\MS)+|\MS|.
\end{equation}

\begin{Theorem}\cite{HHHKR}\label{PrimeDecomposition}
In the above notation, we have
\[J_G=\bigcap_{\MS\subset [n]}P_{\MS}(G).\]
\end{Theorem}

In particular, $J_G$ is a radical ideal and its minimal prime ideals are among $P_{\MS}(G)$ with $\MS\subset [n].$  The following proposition characterizes the sets $\MS$ for which the prime ideal $P_{\MS}(G)$ is minimal.

\begin{Proposition}\label{cpset}\cite[Corollary 3.9]{HHHKR}\label{PrimeDivisor}
$P_{\MS}(G)$ is a minimal prime ideal of $J_G$ if and only if either $\MS=\emptyset$ or $\MS$ is non-empty and for each $i\in \MS,$ $c(\MS\setminus\{i\})<c(\MS)$.
\end{Proposition}

In graph theoretical  terminology, for a connected graph $G$, $P_{\MS}(G)$ is a minimal prime ideal of $J_G$ if and only if $\MS$ is empty or $\MS$ is non-empty and  is a \emph{cut  set} of $G,$ that is, $i$ is a cut vertex of the restriction $G_{([n]\setminus\MS)\cup\{i\}}$ for every $i\in \MS.$ We recall that a vertex $v$ of the  graph $H$ is a \emph{cut vertex} of $H$ if its removing  breaks $H$ into more connected components than $H$ has. Let $\MC(G)$ be the set of all sets $\MS\subset [n]$ such that $P_{\MS}(G)$ is a minimal prime ideal of $J_G.$ Equality 
(\ref{eq:ht}) implies then the following.

\begin{Corollary}\label{unmixed}
Let $G$ be a connected graph on the vertex set $[n]$. Then $J_G$ is unmixed if and only if for every $\MS\in \MC(G)$, $c(\MS)=|\MS|+1.$
\end{Corollary}

A general upper bound for the regularity of binomial edge ideals was first given in \cite{MaMu}, namely, $\reg(S/J_G)\leq n-1,$ and in the same paper it was conjectured that 
$\reg(S/J_G)=n-1$ if and only if $G$ is a path graph. This conjecture was  proved in \cite{KiMa}.

\begin{Theorem}\cite{KiMa}\label{regmax}
Let $G$ be  a graph on $n$ vertices which is not a path. Then $\reg(S/J_G)\leq n-2.$
\end{Theorem}

For a chordal graph $G$, in \cite[Theorem 3.5]{RSK} it was shown that the number $c(G)$ of maximal cliques of $G$ is an upper bound for $\reg(S/J_G).$ 

Recall that a subset $C\subset [n]$ is a \emph{clique} of $G$ if the induced subgraph of $G$ on the vertex set $C$ is a complete graph. The set of cliques of $G$ forms a simplicial complex $\Delta(G)$ called  the \emph{clique complex} of $G.$ Its facets are the maximal cliques of $G.$ By a famous theorem of Dirac (\cite{Dir} or \cite[Section 9.2]{HH10}), a connected graph $G$ is chordal if and only if 
either $G$ is a complete graph or the facets of $\Delta(G)$ can be ordered as $F_1,\ldots,F_c$ such that, for all $i>1,$ $F_i$ is a leaf of the simplicial complex generated by $F_1,\ldots,F_i.$ A \emph{leaf} of a simplicial complex $\Delta$ is a facet of $\Delta$  which has a \emph{branch}, that is, a facet $G$ such that for all facets $F^\prime$ of $\Delta$ with $F^\prime\neq F,$ we have $F^\prime\cap F\subseteq G\cap F. $

\section{A new upper bound for the regularity of binomial edge ideals}
\label{Bound}

In this section, we give a new general upper bound for the regularity of $S/J_G.$ 

\begin{Theorem}\label{NaokiThm}
Let $G$ be a connected graph on $[n]$. Then $\reg(S/J_G)\leq n-\dim \Delta(G).$
\end{Theorem}

When  $G$ is not connected, we derive the following upper bound for the regularity of $S/J_G.$  

\begin{Corollary}\label{NaokiCor}
Let $G$ be a graph on $n$ vertices with the connected components $G_1,\ldots,G_c.$ Then
\[
\reg(S/J_G)\leq n-(\dim \Delta(G_1)+\cdots +\dim \Delta(G_c)).
\]
\end{Corollary}

Let us make some short remarks before proving the above theorem. 
This new bound will be an essential tool in proving Theorem~\ref{thm:licci}. Although for chordal graphs, this bound might be larger than the number of maximal cliques of $G,$ it is a general bound which  is better than $n-1.$ 

In what follows, we will need some notation and known results. If $H$ is a graph and $e\in E(H)$, we denote by $H\setminus e$ the subgraph of $H$ obtained by removing the edge $e$ from $E(H)$ and if $e_1,\ldots,e_m\in E(H),$ we write $H\setminus\{e_1,\ldots,e_m\}$ for the subgraph of $H$ which is obtained by removing the edges $e_1,\ldots,e_m.$ If 
$e=\{i,j\}$ where $i,j$ are vertices of $H$ and $e\not\in E(G)$, then $H\cup e$ is the graph with the same vertex set as $H$ and edge set $E(H)\cup\{e\},$ and $H_e$ is the graph with 
$V(H_e)=V(H)$ and $E(H_e)=E(H)\cup\{\{k,\ell\}: k,\ell\in N(i) \text{ or } k,\ell\in N(j)\}$ where $N(i)$ denotes the set of all neighbors of $i$ in $H.$

The next proposition is a direct consequence of the behavior of the regularity with respect to short  exact sequences; see \cite[Corollary 18.7]{Peeva}. 

\begin{Proposition}\cite[Proposition 2.1]{KiMa}\label{prop3cond}
Let $H$ be a graph on $n$ vertices and $J_H\subset S$ its binomial edge ideal.   Let  $e=\{i,j\}$ be an edge of $H$ and $f_e=x_iy_j-x_jy_i.$ Then, the following inequalities hold:
\begin{itemize}
	\item [\emph{(a)}] $\reg(J_H)\leq \max\{\reg(J_{H\setminus e}), \reg(J_{H\setminus e}:f_e)+1\};$
	\item [\emph{(b)}] $\reg(J_{H\setminus e})\leq \max\{\reg(J_{H}), \reg(J_{H\setminus e}:f_e)+2\};$
	\item [\emph{(c)}] $\reg(J_{H\setminus e}:f_e)+2 \leq \max\{\reg(J_{H\setminus e}), \reg(J_H)+1\}.$
\end{itemize}
\end{Proposition}

In the settings of the above proposition, we have the following.

\begin{Theorem}\cite[Theorem 3.7]{MoSh}\label{propcolon}
\[J_{H\setminus e}:f_e=J_{(H\setminus e)_e} +I_{H,e}
\]
where $I_{H,e}$ is the monomial ideal generated by the set \[\{g_{\pi, t}| \pi: i,i_1,\ldots,i_s,j \text{ is a path between } i \text{ and }j \text{ and }0\leq t\leq s\}\] and 
\[g_{\pi, 0}=x_{i_1}\cdots x_{i_s}, g_{\pi,t}=y_{i_1}\cdots y_{i_t}x_{i_{t+1}}\cdots x_{i_s} \text{ for }  1\leq t\leq s.\]
\end{Theorem}

\begin{proof}[Proof of Theorem~\ref{NaokiThm}]
Clearly, the statement of the theorem follows if we show that for any clique $W\subset [n]$, we have 
\begin{equation}\label{eq:regclique}
\reg (S/J_G)\leq n-|W|+1 \text{ or,  equivalently, } \reg (J_G)\leq n-|W|+2.
\end{equation}

We prove this by induction on $n-|W|.$ If $n=|W|,$ then $G$ is the complete graph on $n$ vertices and it is well known that  $\reg(S/J_G)=1.$

Let $n-|W|>0.$ We proceed with the inductive step. 
For the remaining part of the proof, we need to define the following. For a vertex $v\in V(G),$ we set 
$\alpha_G(v):=\binom{\deg v}{2}-|E(G_{N(v)})|.$ Here, we used the usual  notation $G_U$ for the restriction of $G$ to the subset $U$ of $V(G).$ Obviously, $\alpha_G(v)=0$ if and only if 
$v$ is a simplicial vertex in $G.$ Recall that a  vertex of a graph is called \emph{simplicial} if it belongs to exactly one maximal clique. In addition, for a subset $W\subset V(G),$ we define $\alpha_G(W):=\min\{\alpha_G(v): v\in V(G)\setminus W\}.$ Further on, we proceed by induction on $\alpha_G(W).$\\
\textbf{Step 1.} Let $\alpha_G(W)=0.$ Thus,   there exists a simplicial vertex $v\in V(G)\setminus W.$ Now we consider two cases, namely  $\deg v=1$ and $\deg v\geq 2$.\\
\emph{Case 1.} Let $\deg(v)=1$ and $e=\{v,w\}\in E(G).$ By Proposition~\ref{prop3cond} (a), we have
\[
\reg(J_G)\leq \max\{\reg(J_{G\setminus e}), \reg(J_{G\setminus e}:f_e)+1\}.
\] Therefore, it is enough to show that 
\begin{equation}\label{eq:first}
\reg(J_{G\setminus e})\leq n-|W|+2
\end{equation}
and
\begin{equation}\label{eq:second}
\reg(J_{G\setminus e}:f_e)\leq n-|W|+1.
\end{equation}
Since $\deg(v)=1,$ the vertex $v$ becomes isolated in the graph $G\setminus e,$ thus $\reg(J_{G\setminus e})=\reg(J_{(G\setminus e)\setminus v}).$ So, for showing
inequality (\ref{eq:first}), we simply apply the inductive hypo\-thesis to the graph $(G\setminus e)\setminus v$. For showing inequality 
(\ref{eq:second}), we first apply Theorem~\ref{propcolon} and get
\[
\reg(J_{G\setminus e}:f_e)=\reg (J_{(G\setminus e)_e}),
\] since,  $I_{G,e}=(0)$ because the only path connecting $v$ and $w$ in $G$ is the edge  $\{v,w\}.$ In the graph $(G\setminus e)_e,$ $v$ is an isolated vertex, thus, 
\[
\reg (J_{(G\setminus e)_e})=\reg (J_{((G\setminus e)_e)\setminus v}).
\] Now we can apply again the inductive hypothesis for $(G\setminus e)_e\setminus v$ and  obtain
\[
\reg (J_{(G\setminus e)_e\setminus v})\leq (n-1)-|W|+2=n-|W|+1.
\] Therefore, Case 1 is completed.\\
\emph{Case 2.} Let $v$ be a simplicial vertex of $\deg(v)=t\geq 2.$ Before discussing this case, we prove the following. \\
\textbf{Claim}. Assume that there exists $v\in V(G)\setminus W$ a simplicial vertex with $\deg(v)\geq 2.$ Let $e$ be an edge of $G$ which contains $v.$ Then
\[
\reg(J_{G\setminus e}:f_e)\leq n-|W|+1.
\]
\begin{proof}[Proof of the Claim]
Let $\deg(v)=t,$ let $N_G(v)=\{v_1,\ldots,v_t\}$ be the set of neighbors of $v$ in $G,$ and set $e_i=\{v,v_i\}$ for $1\leq i\leq t.$ We may assume that $e=e_t$ and let us  consider the monomial ideal $I_{G,e}$ from Theorem~\ref{propcolon}.
Since $v$ is a simplicial vertex, for any $1\leq i\leq t-1,$ $v_t,v_i,v$ is a path in $G,$ thus $x_{v_i},y_{v_i}\in I_{G,e}$ for all $1\leq i\leq t-1.$ Moreover, every path from $v$ to $v_t$ 
must pass through some neighbor $v_i$ with $1\leq i\leq t-1.$ This implies that 
\[
I_{G,e}=(x_{v_i},y_{v_i}:1\leq i\leq t-1).
\]
By Theorem~\ref{propcolon}, we get
\[
J_{G\setminus e}:f_e=J_{(G\setminus e)_e} +(x_{v_i},y_{v_i}:1\leq i\leq t-1).
\] Set $H:=(G\setminus e)_e.$ Then
\[
J_{G\setminus e}:f_e=J_{H_{[n]\setminus\{v_1,\ldots,v_{t-1}\}}} +(x_{v_i},y_{v_i}:1\leq i\leq t-1),
\] because the binomial generators of $H=J_{(G\setminus e)_e}$ corresponding to the edges which contain some $v_i$ with $1\leq i\leq t-1$ are contained in $I_{G,e}.$ Since $v$ becomes an isolated 
vertex in $H_{[n]\setminus\{v_1,\ldots,v_{t-1}\}},$  we get 
\[
J_{G\setminus e}:f_e=J_{H_{[n]\setminus\{v,v_1,\ldots,v_{t-1}\}}} +(x_{v_i},y_{v_i}:1\leq i\leq t-1),
\]
which implies that 
\[
\reg (J_{G\setminus e}:f_e)=\reg (J_{H_{[n]\setminus\{v,v_1,\ldots,v_{t-1}\}}}).
\]
The graph $H_{[n]\setminus\{v,v_1,\ldots,v_{t-1}\}}$ has $n-t$ vertices and the clique $W\setminus \{v,v_1,\ldots,v_{t-1}\},$ thus we  may apply  the inductive hypothesis because 
\[(n-t)-|W\setminus \{v,v_1,\ldots,v_{t-1}\}|\leq n-t-|W|+t-1= n-|W|-1.\] Therefore, we get
\[
\reg (J_{G\setminus e}:f_e)=\reg (J_{H_{[n]\setminus\{v,v_1,\ldots,v_{t-1}\}}})\leq (n-t)-|W\setminus \{v,v_1,\ldots,v_{t-1}\}|+2\leq n-|W|+1,
\] and the claim is proved.
\end{proof} We now go back to the discussion of \emph{Case 2.}
Let $N_G(v)=\{v_1,\ldots, v_t\}$ be the set of the neighbors of $v$ in $G$ and $e_i=\{v,v_i\}$ for $1\leq i\leq t.$ By Proposition~\ref{prop3cond} and the Claim, 
we have
\[
\reg(J_G)\leq \max\{\reg(J_{G\setminus e_1}), \reg(J_{G\setminus e_1}:f_{e_1})+1\}\leq \max\{\reg(J_{G\setminus e_1}),n-|W|+2\}. 
\] Applying the same argument to $G\setminus e_1,$ we obtain
\[
\reg(J_G)\leq \max\{\reg(J_{G\setminus \{e_1,e_2\}}),n-|W|+2\}. 
\]
After $t-1$ steps, we get
\[
\reg(J_G)\leq \max\{\reg(J_{G\setminus \{e_1,e_2,\ldots,e_{t-1}\}}),n-|W|+2\}. 
\]
In the graph $G\setminus \{e_1,e_2,\ldots,e_{t-1}\}$, we have $\deg(v)=1.$ Consequently, by Case 1, we derive that $\reg(J_G)\leq n-|W|+2$ which completes the proof of Step 1.

Now we proceed to prove the inductive step on $\alpha_G(W)$.\\
\textbf{Step 2.} Let $\alpha_G(W)>0.$ This implies that there exists a non-simplicial vertex $v\in V(G)\setminus W$ such that $\alpha_G(W)=\alpha_G(v).$ Moreover, since $v$ is not simplicial, there exist $v_1,v_2\in N_G(v)$ such that $e=\{v_1,v_2\}\not\in E(G).$ By Proposition~\ref{prop3cond} (b) where $H=G\cup e,$ it follows
\begin{equation}\label{eq:third}
\reg(J_G)\leq \max\{\reg(J_{G\cup e}), \reg(J_G:f_e)+2\}.
\end{equation}
By the definition of $\alpha_G(v),$ we have $\alpha_{G\cup e}(v)=\alpha_G(v)-1,$ therefore $\alpha_{G\cup e}(W)\leq \alpha_G(W)-1.$ By induction on $\alpha_G(W)$, we then derive that 
\[
\reg(J_{G\cup e})\leq n-|W|+2.
\] In order to complete this last step, by using (\ref{eq:third}), it is enough to show that 
\begin{equation}\label{eq:last}
\reg(J_G:f_e)+2\leq n-|W|+2.
\end{equation} 
By Theorem~\ref{propcolon}, we have 
\begin{equation}\label{eq:fourth}
J_G:f_e=J_{G_e}+I_{G\cup e, e}.
\end{equation} Since $v_1,v,v_2$ is a path, the variables $x_v,y_v$ belong to $I_{G\cup e, e}.$ This implies that 
\[
I_{G\cup e, e}=(x_v,y_v)+I_{(G\setminus v)\cup e,e}.
\]
By replacing $I_{G\cup e, e}$ in equality (\ref{eq:fourth}), we can rewrite it as
\[
J_G:f_e=J_{(G\setminus v)_e}+ I_{(G\setminus v)\cup e,e} +(x_v,y_v).
\] This implies that 
\[
\reg(J_G:f_e)=\reg(J_{(G\setminus v)_e}+ I_{(G\setminus v)\cup e,e}).
\] On the other hand, by Theorem~\ref{propcolon} applied for $G\setminus v,$ we get
\[
\reg(J_{(G\setminus v)_e}+ I_{(G\setminus v)\cup e,e})=\reg(J_{G\setminus v}:f_e),
\] thus, 
\[
\reg(J_G:f_e)=\reg(J_{G\setminus v}:f_e).
\] Next, by Proposition~\ref{prop3cond}, (c) we have
\[
\reg(J_{G\setminus v}:f_e)+2\leq \max\{\reg(J_{G\setminus v}), \reg(J_{(G\setminus v)\cup e}) +1\}.
\] By the inductive hypothesis on $n-|W|,$ we have 
\[
\reg(J_{G\setminus v})\leq (n-1)-|W|+2=n-|W|+1,
\] and 
\[
\reg(J_{(G\setminus v)\cup e})+1\leq (n-1)-|W|+3=n-|W|+2.
\] Consequently, we proved inequality (\ref{eq:last}) and this completes Step 2 and the whole proof of the theorem.
 \end{proof}

\section{Licci binomial edge ideals}
\label{Liccisec}

As in the previous section, let $G$ be a simple graph on the vertex set $[n]$ and $S=K[x_1,\ldots, x_n,y_1,\ldots,y_n]$ the polynomial ring over a field $K.$ Let $\mm$ be the maximal graded ideal of $S$ and set $R=S_{\mm}.$

We recall the notion of decomposable graphs from \cite{HeRi}.

\begin{Definition}{\em
A connected graph $G$ is called \emph{decomposable} if there exists two subgraphs $G_1$ and $G_2$ of $G$ such that $G=G_1\cup G_2$ with $V(G_1)\cap V(G_2)=\{v\}$ where $v$ is a simplicial vertex in $G_1$ and 
$G_2.$ In this case we say that $G$ is decomposable in the vertex $v.$ Otherwise, the graph $G$ is called indecomposable.
}
\end{Definition}

As it was proved in \cite{HeRi}, if $G$ is decomposable, then $\reg(S/J_G)=\reg S_1/J_{G_1}+\reg S_2/J_{G_2}$ where $S_i=K[\{x_j,y_j:j\in V(G_i)\}]$ for $i=1,2.$ Moreover, by \cite[Theorem 2.7]{RaRi}, 
$J_G$ is Cohen-Macaulay if and only if $J_{G_1}$ and $J_{G_2}$ are Cohen-Macaulay.
\medskip

Before proving the main theorem of this section, we state some lemmas which are useful in what follows. 

\begin{Lemma}\label{lm:decompose}
Let $G$ be a decomposable graph as $G=G_1\cup G_2$ with $|V(G_i)|=n_i$ for $i=1,2$ and let $S_i=K[\{x_j,y_j\}:j\in V(G_i)]$ for $i=1,2.$
If $\reg(S/J_G)=n-2,$ then $\reg(S_1/J_{G_1})=n_1-2$ and $G_2$ is a path or $\reg(S_2/J_{G_2})=n_2-2$ and $G_1$ is a path.
\end{Lemma} 

\begin{proof}
We have  \[n-2=\reg(S/J_G)=\reg(S_1/J_{G_1})+\reg(S_2/J_{G_2})\leq (n_1-1)+(n_2-1)=n-1.\] This
implies that $\reg(S_1/J_{G_1})=n_1-2$  and $\reg(S_2/J_{G_2})=n_2-1,$ which means that $G_2$ is a path by Theorem~\ref{regmax}, or
$\reg(S_2/J_{G_2})=n_2-2$ and $\reg(S_1/J_{G_1})=n_1-1$, that is, $G_1$ is a path. 
\end{proof}

\begin{Lemma}\label{lm:4neighbors}
Let $G$ be a connected  graph on the vertex set $[n]$. Suppose that  $G$ has a cut vertex $v$ with $\deg_G(v)\geq 4$. Then $\reg(S/J_G)\leq n-3.$
\end{Lemma}

\begin{proof}
Since $v$ is  a cut vertex of $G$, by \cite[Lemma 4.8]{Ohtani}, we get 
\[
J_G=J_{G_v} \cap (J_{G\setminus v}+(x_v,y_v))
\] where $G_v$ is the graph on  $V(G_v)=V(G)$  with the edge set  \[E(G_v)=E(G)\cup \{\{u,w\}: u,w\in N_G(v)\}.\]
Consequently, we have the following exact sequence
\[
0\to \frac{S}{J_G}\to \frac{S}{J_{G_v}}\bigoplus\frac{S}{J_{G\setminus v}+(x_v,y_v)}\to \frac{S}{J_{G_v\setminus v}+(x_v,y_v)}\to 0,
\] since $J_{G_v}+(J_{G\setminus v}+(x_v,y_v))=J_{G_v\setminus v}+(x_v,y_v).$
From this exact sequence we obtain
\begin{equation}\label{ineq:sequ}
\reg\frac{S}{J_G}\leq \max\{\reg\frac{S}{J_{G_v}}, \reg\frac{S}{J_{G\setminus v}+(x_v,y_v)}, \reg \frac{S}{J_{G_v\setminus v}+(x_v,y_v)}+1\}.
\end{equation}
By our assumption, $v$ has at least 4 neighbors in $G.$ Therefore, in $G_v$ we have a maximal clique with at least $5$ vertices. By 
Theorem~\ref{NaokiThm}, we have $\reg(S/J_{G_v})\leq n-4.$ 
The graph $G\setminus v$ has $n-1$ vertices and at least  two connected components, say $G_1,\ldots, G_c$ with $c\geq 2,$ because $v$ is a cut vertex of $G$. Let 
$S^\prime=K[\{x_j,y_j\}: j\in [n]\setminus\{v\}].$ Then 
\[
\frac{S^\prime}{J_{G\setminus v}}\cong \frac{S_1}{J_{G_1}}\bigotimes_K\cdots \bigotimes_K \frac{S_c}{J_{G_c}}
\] where $S_i=K[\{x_j,y_j\}: j\in V(G_i)]$ for $i=1\ldots,c.$ This implies that
\[
\reg(S/J_{G\setminus v}+(x_v,y_v))=\reg(S^\prime/J_{G\setminus v})=\sum_{i=1}^c\reg(S_i/J_{G_i})\]  \[\leq \sum_{i=1}^c(|V(G_i)|-1)=(n-1)-c
\leq n-3.\]
If $v$ has at least $4$ neighbors in $G,$ then the graph $G_v\setminus v$ has a maximal clique with at least $4$ vertices, thus, by 
Theorem~\ref{NaokiThm}, we get \[\reg(S/J_{G_v\setminus v}+(x_v,y_v)))=\reg(S^\prime/J_{G_v\setminus v})\leq (n-1)-3=n-4.\]

Therefore, from  inequality (\ref{ineq:sequ}), we  get $\reg(S/J_G)\leq n-3.$
\end{proof}

\begin{Lemma}\label{lm:gap}
Let $G$ be a connected indecomposable graph on $n\geq 4$ vertices with the following properties:
\begin{itemize}
	\item [\emph{(a)}] $J_G$ is unmixed;
	\item [\emph{(b)}] $G$ has a vertex $v$ with exact two neighbors $u_1,u_2$ and $\{u_1,u_2\}\in E(G).$
\end{itemize}
Then $\reg(S/J_G)\leq n-3.$
\end{Lemma}

\begin{proof} If $n=4,$ then there are only two graphs which satisfy the condition (b), namely two triangles which share an edge and a triangle with an edge attached to one of its vertices; see Figure~\ref{fig:n=4}.

\begin{figure}[hbt]
\begin{center}
\psset{unit=0.9cm}
\begin{pspicture}(-4,-2)(5,2)
\psline(1,0)(5,0)
\pspolygon(1,0)(2,2)(3,0)

\rput(1,0){$\bullet$}
\rput(3,0){$\bullet$}
\rput(2,2){$\bullet$}
\rput(5,0){$\bullet$}

\psline(-3,0)(-1,0)
\pspolygon(-3,0)(-2,2)(-1,0)(-2,-2)

\rput(-3,0){$\bullet$}
\rput(-2,2){$\bullet$}
\rput(-1,0){$\bullet$}
\rput(-2,-2){$\bullet$}

\end{pspicture}
\end{center}
\caption{4 vertices}\label{fig:n=4}
\end{figure}
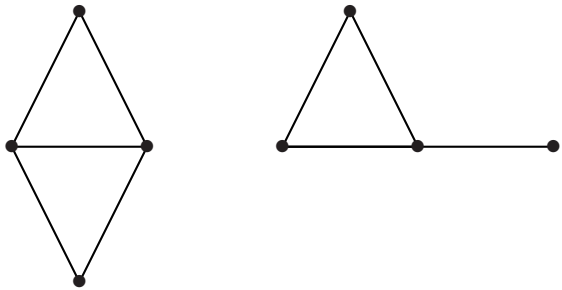

The first graph does not satisfy the condition (a), while the second graph is decomposable. Thus, we may consider $n\geq 5.$

Let us consider an indecomposable graph $G$ with $n\geq 5$ vertices satisfying the conditions (a) and (b).  We claim that $\deg u_1\geq 4$ or $\deg u_2\geq 4.$ Let us assume that this is not the case, thus $\deg u_1\leq 3$ and 
$\deg u_2\leq 3.$ Since $G$ is indecomposable, it follows that $\deg u_1= 3,$  $\deg u_2= 3,$ and there exists a path connecting $u_1$ and $u_2$ different from the edge $\{u_1,u_2\}$ and the path $u_1,v,u_2.$ But, in this case, the set $\MS=\{u_1,u_2\}$ is a cut set 
of $G$ with $c(\MS)=|\MS|,$ which is impossible since $J_G$ is an unmixed ideal. 

Without loss of generality, we may assume that $\deg u_2\geq 4.$

We set $e=\{u_1,v\}.$ By 
Proposition~\ref{prop3cond} (a),
we have
\begin{equation}\label{eq:removee}
\reg\frac{S}{J_G}\leq \max\left\{\reg\frac{S}{J_{G\setminus e}},\reg\frac{S}{J_{G\setminus e}:f_e}+1\right\}.
\end{equation}
 
In the graph $G\setminus e,$ $u_2$ is a cut vertex with at least $4$ neighbors. Thus, by Lemma~\ref{lm:4neighbors}, it follows that 
\[
\reg\frac{S}{J_{G\setminus e}}\leq n-3.
\] Now we look at $J_{G\setminus e}:f_e.$ By applying Theorem~\ref{propcolon}, we obtain
\[
J_{G\setminus e}:f_e=J_{(G\setminus e)_e}+(x_{u_2},y_{u_2})
\] since all the paths connecting $u_1$ and $v$ pass trough $u_2.$ Therefore, since $v$ becomes an isolated vertex in the graph 
$(G\setminus e)_e\setminus u_2,$ we get
\[
\reg\frac{S}{J_{G\setminus e}:f_e}=\reg\frac{S}{J_{(G\setminus e)_e}+(x_{u_2},y_{u_2})}=
\reg\frac{S^\prime}{J_{(G\setminus e)_e\setminus \{u_2,v\}}} 
\] where $S^\prime=K[\{x_j,y_j\}:j\in [n]\setminus\{u_2,v\}].$ If the graph $(G\setminus e)_e\setminus \{u_2,v\}$
is a path, as $\deg u_2\geq 4,$ the graph $G$ looks like in Figure~\ref{fig:notpath}, that is, there are some edges connecting $u_2$ 
to some vertices of the  the path $(G\setminus e)_e\setminus \{u_2,v\}$ different from $u_1.$ But then $J_G$ is not unmixed since 
$\MS=\{u_1,u_2\}$ is a cut set of $G$ with $c(\MS)=|\MS|,$ a contradiction. 
Therefore, the graph $(G\setminus e)_e\setminus \{u_2,v\}$
is not a path.
 Thus, by Theorem~\ref{regmax}, we obtain
\[
\reg\frac{S}{J_{G\setminus e}:f_e}=\reg\frac{S^\prime}{J_{(G\setminus e)_e\setminus \{u_2,v\}}}\leq (n-2)-2=n-4,
\] which implies that 
\[
\reg\frac{S}{J_{G\setminus e}:f_e}+1\leq n-3.
\] and the proof of the lemma is completed. 

\begin{figure}[hbt]
\begin{center}
\psset{unit=1cm}
\begin{pspicture}(-6,0)(3,4)



\psline(-3,0)(-3,3)
\psline(-3,0)(-1,3)
\psline(-1,3)(-3,2)
\pspolygon(-3,3)(-1,3)(-2,4)

\rput(-3,0){$\bullet$}
\rput(-3,1){$\bullet$}
\rput(-3,2){$\bullet$}
\rput(-3,3){$\bullet$}
\rput(-2,4){$\bullet$}
\rput(-1,3){$\bullet$}

\rput(-2,4.3){$v$}
\rput(-0.6,3){$u_2$}
\rput(-3.4,3){$u_1$}
\rput(-2.6,3.7){$e$}
\end{pspicture}
\end{center}
\caption{The graph $G$ when  $(G\setminus e)_e\setminus \{u_2,v\}$
is a path}\label{fig:notpath}
\end{figure}
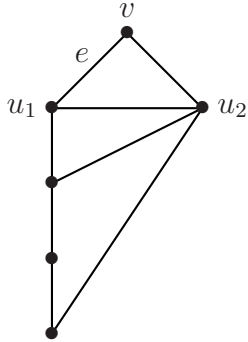
\end{proof}

We can now state the main result of this section. 

\begin{Theorem}\label{thm:licci}
Let $G$ be a connected graph on the vertex set $[n].$ Then the following statements are equivalent:
\begin{itemize}
	\item [\emph{(i)}]  $(J_G)_{\mm}\subset R$ is licci.
	\item [\emph{(ii)}] $J_G$ is  Cohen-Macaulay and $n-2\leq \reg(S/J_G)\leq n-1.$
	\item [\emph{(iii)}] $G$ is a path graph or it is isomorphic to one of the graphs depicted in Figure~\ref{fig:licci} where $r,s,t$ are non-negative integers. In other words, $G$ is a triangle with possibly some paths connected to some of its vertices.
\end{itemize}
\end{Theorem}

\begin{figure}[hbt]
\begin{center}
\psset{unit=0.9cm}
\begin{pspicture}(-5,0)(4,5)
\psline(-4,0)(4,0)
\pspolygon(-1,0)(1,0)(0,2)
\psline(0,2)(0,5)
\rput(-4,0){$\bullet$}
\rput(-3,0){$\bullet$}
\rput(-2,0){$\bullet$}
\rput(-1,0){$\bullet$}
\rput(1,0){$\bullet$}
\rput(2,0){$\bullet$}
\rput(3,0){$\bullet$}
\rput(4,0){$\bullet$}
\rput(0,2){$\bullet$}
\rput(0,3){$\bullet$}
\rput(0,4){$\bullet$}
\rput(0,5){$\bullet$}

\rput(1.5,0.3){$e_1$}
\rput(2.5,0.3){$\cdots$}
\rput(3.5,0.3){$e_r$}

\rput(-1.5,0.3){$e^{\prime}_1$}
\rput(-2.5,0.3){$\cdots$}
\rput(-3.5,0.3){$e^{\prime}_s$}

\rput(0.3,2.5){$e^{\prime\prime}_1$}
\rput(0.3,3.5){$\vdots$}
\rput(0.3,4.5){$e^{\prime\prime}_t$}

\end{pspicture}
\end{center}
\caption{Licci graphs}\label{fig:licci}
\end{figure}
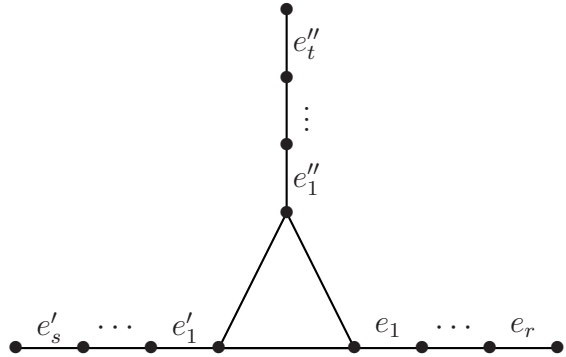

\begin{proof} (i) $\Rightarrow$ (ii). Let $(J_G)_{\mm}\subset R$ be licci.
By Theorem~\ref{HUineq}, it follows that 
$\reg(S/J_G)\geq \height(J_G)-1.$ Since $J_G$ is Cohen-Macaulay, thus unmixed, we have $\height(J_G)=\height P_{\emptyset}(G)=n-1,$ by (\ref{eq:ht}). Therefore, if $G$ is connected and $(J_G)_{\mm}$ 
is licci, then $J_G$ is Cohen-Macaulay and $\reg(S/J_G)\geq n-2.$ But we know from \cite{MaMu} that $\reg(S/J_G)\leq n-1$. 

Let us prove that (ii) $\Rightarrow$ (iii). Since, by Theorem~\ref{regmax}, we have $\reg(S/J_G)=n-1$ if and only if $G$ is a path graph, it remains to consider $\reg(S/J_G)=n-2.$  By using Lemma~\ref{lm:decompose},
we may reduce the problem to considering only the case when $G$ is indecomposable. Therefore, in order to get (iii), by taking into account Lemma~\ref{lm:decompose}, it is enough to show that there is no indecomposable graph $G$ with 
$|V(G)|\geq 4$ such that $J_G$ is Cohen-Macaulay and $\reg(S/J_G)=n-2.$ There is no such graph among those with $4$ vertices. Thus, we may consider $n=|V(G)|\geq 5.$

Let us assume that such a graph does exist.  By \cite[Remark 5.3]{BMS}, since $J_G$ is Cohen-Macaulay, the graph $G$ must have a cut vertex, say $v.$ Since $G$ is indecomposable, $v$ has at least $3$ neighbors in $G.$ If $v$ has at least $4$ neighbors, by Lemma~\ref{lm:4neighbors}, it follows that $\reg(S/J_G)\leq n-3,$ a contradiction. Thus, $v$ has exactly $3$ neighbors, say $w,u_1,u_2.$ 
 Since $G$ is indecomposable and 
$v$ is a cut vertex in $G,$ it follows that none of the edges $\{u_1,u_2\},\{u_1,w\},\{u_2,w\}$ belongs to $E(G).$ On the other hand, as 
$J_G$ is unmixed, the graph $G\setminus v$ has exactly two connected components, say $G_1$ and $G_2.$
 We may assume that $u_1,u_2$ are vertices in $G_1$ and $w$ is a vertex in 
$G_2$. Let $e=\{v,w\}.$ By Proposition~\ref{prop3cond} (a), we have
\begin{equation}\label{eq:remove}
n-2=\reg\frac{S}{J_G}\leq \max\left\{\reg\frac{S}{J_{G\setminus e}},\reg\frac{S}{J_{G\setminus e}:f_e}+1\right\}.
\end{equation}
We observe that $G\setminus e$ has two connected components, namely $G^\prime$ with $V(G^\prime)=V(G_1)\cup\{v\}$ and 
$E(G^\prime)=E(G_1)\cup\{\{u_1,v\},\{u_2,v\}\}$ and $G^{\prime\prime}=G_2.$ Obviously, $G^\prime$ is not a path graph since $G_1$ is 
connected, thus there exists  at least one path connecting $u_1$ and $u_2$ in $G_1$ which does not contain $v$ and is not the edge 
$\{u_1,u_2\}.$ On the other hand, if 
$G_2$ does not consist only of the isolated vertex $w,$ then $G_2$ cannot be a path since the graph $G$ is indecomposable. Let $S^\prime=K[\{x_j,y_j\}: j\in V(G^\prime)]$ and 
$S^{\prime\prime}=K[\{x_j,y_j\}: j\in V(G^{\prime\prime})].$ Then, by Theorem~\ref{regmax}, we have
\[
\reg\frac{S^\prime}{J_{G^\prime}}+\reg\frac{S^{\prime\prime}}{J_{G^{\prime\prime}}}\leq (|V(G^\prime)|-2)+(|V(G^{\prime\prime})|-2)= n-4.
\] Therefore, 
\[
\reg\frac{S}{J_{G\setminus e}}=\reg\frac{S^\prime}{J_{G^\prime}}+\reg\frac{S^{\prime\prime}}{J_{G^{\prime\prime}}}< n-3. 
\] If $G_2$ consist only of the isolated vertex $w,$ then we get
\[
\reg\frac{S}{J_{G\setminus e}}=\reg\frac{S^\prime}{J_{G^\prime}}\leq |V(G^\prime)|-2= n-3. 
\] Thus, in any case we have
\begin{equation}\label{eq:Ge}
\reg\frac{S}{J_{G\setminus e}}\leq n-3.
\end{equation}
Now we look at the term $\reg(S/J_{G\setminus e}:f_e)$ of inequality (\ref{eq:remove}). By Theorem~\ref{propcolon}, it follows that 
$J_{G\setminus e}:f_e=J_{(G\setminus e)_e}$ since there is no path in $G$ connecting $v$ and $w$ except the edge $e=\{v,w\}.$ This is due 
to the fact that when we remove the cut vertex $v$ from $G,$ we get two connected components by the unmixedness of $J_G$. The graph $(G\setminus e)_e$ consists as well of two
connected components, say $H_1$ which contains $v$  and $H_2$ which contains $w.$ If $H_2$ contains some other vertices together with $w,$ 
then $H_2$ cannot be a path since $G$ is indecomposable. The component $H_1$ is not a path since it contains at least the triangle with 
vertices $u_1,u_2,v.$ Therefore, if $S_i=K[\{x_j,y_j\}: j\in V(H_i)]$ for $i=1,2,$ by Theorem~\ref{regmax}, we obtain
\[
\reg\frac{S}{J_{G\setminus e}:f_e}=\reg\frac{S}{J_{(G\setminus e)_e}}=\reg\frac{S_1}{J_{H_1}}+\reg\frac{S_2}{J_{H_2}}\leq (|V(H_1)|-2)+(|V(H_2)|-2)=n-4.
\] This inequality and  (\ref{eq:Ge}) contradicts inequality (\ref{eq:remove}).

It remains to analyze the case when $H_2$ consists of the isolated vertex $w.$ In this case we have
\begin{equation}\label{eq:h1}
\reg\frac{S}{J_{(G\setminus e)_e}}=\reg\frac{S_1}{J_{H_1}}.
\end{equation} We claim that $H_1$ satisfies the conditions of Lemma~\ref{lm:gap}. Clearly, $H_1$ satisfies the condition (b). It 
remains to prove that $J_{H_1}$ is an unmixed ideal because if $H_1$ is decomposable in $u_1$ or $u_2$, then $G$ is decomposable, and this is impossible by our hypotheses on $G$. We first observe that any non-empty cut set of $H_1$ does not contain the vertex $v$ which is a simplicial vertex in $H_1$. Let us assume that there exists a non-empty cut set $\MS\subset V(H_1)$ such that 
$c_{H_1}(\MS)\neq |\MS|+1.$ The set $\MS$ is obviously a cut set for the graph $G$ as well. Moreover,  if $H_1,\ldots,H_{c_{H_1}(\MS)}$ are the connected components of the restriction of $H_1$ to the vertex set $V(H_1)\setminus \MS,$ with $v\in V(H_1)$, then the connected components of the restriction of $G$ to $V(G)\setminus \MS$ are $H_1\cup\{v,w\},H_2,\ldots,H_{c_{H_1}(\MS)}$. Hence
$c_G(\MS)=c_{H_1}(\MS)\neq |\MS|+1,$ a contradiction to the unmixedness of $J_G.$ Since $H_1$ is a graph on $n-1$ vertices which satisfies 
the conditions of Lemma~\ref{lm:gap}, we get $\reg(S_1/J_{H_1})\leq (n-1)-3=n-4. $ Thus, we have proved that
\[
\reg\frac{S}{J_{G\setminus e}:f_e}=\reg\frac{S}{J_{(G\setminus e)_e}}=\reg\frac{S_1}{J_{H_1}}\leq n-4. 
\] This inequality together with   (\ref{eq:Ge}) contradicts inequality (\ref{eq:remove}) and the proof of (ii) $\Rightarrow$ (iii) is completed.

\vspace{.2in}
Finally, we prove the implication (iii) $\Rightarrow$ (i).

As it was observed in the proof of \cite[Proposition 3]{HeRi}, if $G=G_1\cup G_2$ is a decomposable graph, then we have 
$\Tor_i(S/J_{G_1},S/J_{G_2})=0$ for all $i>0.$ In particular, it follows that $J_{G_1}$ and $J_{G_2}$ are transversal ideals in the sense of \cite[Section 2]{Joh}. Now, let $G_1$ be a triangle with the vertices $v_1,v_2,v_3.$ Then $J_{G_1}$ is a Cohen-Macaulay ideal of height 
$2,$ thus it is licci by \cite{PeSz}. If we attach a path $G_2$ to $G_1$ in one of its vertices, say $v_1,$ the resulting graph $G$ is decomposable in $v_1$ and $J_{G_2}$ is a complete intersection ideal. According to \cite[Theorem 2.6]{Joh}, it follows that 
$(J_G)_\mm$ is a licci ideal. We repeat this argument by attaching  a path in the vertex $v_2$ to $G$ and, next another path in the vertex $v_3.$ In each step, we get a licci ideal. 
\end{proof}

An immediate consequence of the above theorem is the following.

\begin{Corollary}\label{cor:bipartite}
Let $G$ be a connected bipartite graph. Then the  ideal $(J_G)_\mm\subset R=S_\mm$ is licci if and only if $G$ is  a path graph, or equivalently, 
$J_G$ is a complete intersection.
\end{Corollary}

We now turn to the disconnected graphs.

\begin{Proposition}\label{pr:disconnect}
Let $G$ be a graph with the connected components $G_1,\ldots,G_c$ where $c\geq 2.$ Then $(J_G)_\mm\subset R=S_\mm$ is licci if and only if either all the connected components of $G$ are paths or  one component of $G$ is isomorphic to a graph of Figure~\ref{fig:licci} and all the other components are paths. 
\end{Proposition}

\begin{proof} We first remark that, by \cite[Theorem 2.6]{Joh}, if the components of $G$ satisfy the conditions of the proposition, then 
$(J_G)_\mm$ is licci since the ideals $J_{G_i}$ are pairwise transversal by \cite[Lemma~3.1]{HNTT}.

For the converse, let $(J_G)_\mm$ be a licci ideal. Then $J_G$ is Cohen-Macaulay which implies that all the ideals $J_{G_i}$ are Cohen-Macaulay
and 
\[
\reg(S/J_G)\geq \height(J_G)-1=\height(J_{G_1})+\cdots +\height(J_{G_c})-1=n-c-1.
\]
On the other hand, we have
\[
\reg(S/J_G)=\sum_{i=1}^c \reg(S_i/J_{G_i})\leq \sum_{i=1}^c(|V(G_i)|-1)=n-c.
\] Here $S_i=K[\{x_j,y_j\}:j\in V(G_i)]$ for $1\leq i\leq c.$
The above inequalities imply that $\reg(S/J_G)=n-c$ or $\reg(S/J_G)=n-c-1.$ In the first case, it follows that 
$\reg(S_i/J_{G_i})=|V(G_i)|-1$ for all $i,$ which implies that all the connected components of $G$ are path graphs.

Let $\reg(S/J_G)=n-c-1.$ This means that for one of the connected components, say $G_1,$ we have $\reg(S_1/J_{G_1})=|V(G_1)|-2$ and all the other components of $G$ are path graphs. Then, by Theorem~\ref{thm:licci}, it follows that $G_1$ is isomorphic to one of the graphs displayed in Figure~\ref{fig:licci}.
\end{proof}

\section{Licci binomial edge ideals of chordal graphs}
\label{chordal}

In this section we show that if we restrict to chordal graphs, we may relax the condition (ii) in Theorem~\ref{thm:licci}, namely, we may ask that $J_G$ is only unmixed instead of being Cohen-Macaulay. Before proving the main theorem of this section, we need a preparatory result. We first recall that for a graph $G,$ $c(G)$ denotes the number of maximal cliques of $G$, that is, the number of facets of the clique complex $\Delta(G).$

\begin{Lemma}\label{codim1}
Let $G$ be a connected chordal  graph with $n$ vertices. Then $c(G)=n-2$ if and only if  the following conditions hold:
\begin{itemize}
	\item[\emph{(i)}] the maximal cliques of $G$ have at most $3$ vertices; 
	\item[\emph{(ii)}] $G$ has at least one maximal clique with $3$ vertices;
	\item[\emph{(iii)}] $G$ has exactly one maximal clique with $3$ vertices or, for any two triangles $F_1,F_2$ of $\Delta(G),$ there is a sequence of triangles $F_1=F_{i_1},\ldots, F_{i_r}=F_2$ such that for any $1\leq j\leq r-1,$ $F_{i_j}$ and 
	$F_{i_{j+1}}$ share an edge.
\end{itemize}
\end{Lemma}

\begin{proof} Let $c(G)=n-2.$ Then (i) follows by \cite[Proposition 3.1]{RSK}. If $G$ has no maximal clique with $3$ vertices, then $G$ is a tree, thus $c(G)=n-1,$ contradiction. Therefore, condition (ii) holds. 

We prove (iii) by induction on $n.$ Since $G$ is chordal, by Dirac's theorem, we may order 
the facets of $\Delta(G)$ as $F_1,\ldots,F_c$ where $c=c(G)$ such that $F_i$ is a leaf of $\langle F_1,\ldots,F_i\rangle$ for all $i.$ 
If $F_c$ is an edge, say $F_c=\{v,w\}$ with $\deg w=1,$ then the graph $G\setminus w$ has $n-1$ vertices and $n-3$ cliques, thus, by induction, it satisfies (iii), and, consequently, $G$ satisfies (iii) as well. 

Let  $F_c$ be a triangle with the vertices $u,v,w$ and assume that $F_j$ with $j<c$ is a branch of $F_c.$ If $F_j\cap F_c$ consists of just one vertex, say $F_j\cap F_c=\{v\},$ then the subgraph $G^\prime=G\setminus\{u,w\}$ has $n-2$ vertices and $n-3$ maximal cliques, 
therefore $G^\prime$ is a tree. This implies that $\Delta(G)$ has exactly one facet with $3$ elements, and condition (iii) is automatically fulfilled. Let us now assume that the branch $F_j$ intersects  $F_c$ in the edge $\{v,w\}.$ We consider the graph 
$G\setminus u.$ This is a graph on $n-1$ vertices with $n-3$ maximal cliques, thus, by the inductive hypothesis, it satisfies (iii).
Let us choose two triangles $F,F^\prime$ in $\Delta(G).$ If they are facets in $\Delta(G\setminus u),$ then they satisfy (iii). Otherwise,
we may assume that $F^\prime=F_c.$ But then, by the inductive hypothesis on $G\setminus u$ there is a sequence of triangles 
$F=F_{i_1},\ldots, F_{i_r}=F_j$ such that for any $1\leq s\leq r-1,$ $F_{i_s}$ and 
	$F_{i_{s+1}}$ share an edge. Then the sequence $F=F_{i_1},\ldots, F_{i_r}=F_j,F_{i_{r+1}}=F_c$ satisfies the required condition for $G$.
	
	For the converse, let us assume that $G$ is a connected chordal  graph with $n$ vertices.\ which satisfies the three  conditions of the statement. By condition (ii) and \cite[Proposition 3.1]{RSK}, it follows that $c(G)\leq n-2.$
	
	Let us assume that there exists a connected chordal graph $G$ satisfying conditions (i)--(iii) and such that $c(G)<n-2$ and choose one with the minimal number of vertices. We consider again the leaf order $F_1,\ldots,F_c$ on the facets of $\Delta(G)$ and take 
	$F_j$ with $j<c$  a branch of $F_c.$ If $F_c$ is an edge, $F_c=\{v,w\}$ with $\deg w=1$, then the graph $G\setminus w$ has $n-1$ vertices
	and satisfies conditions (i)--(iii), thus, by our assumption on $G$ we have $c(G\setminus w)=n-3,$ which implies that $c(G)=n-2,$ contradiction. 
	
	If $F_c$ is a triangle, $F_c=\{u,v,w\},$ and $F_j$ intersects $F_c$ in just one vertex, say $v,$ then we have the following cases.

\emph{	Case 1. }The facet $F_c$ is the only triangle in $\Delta(G).$ Then, the subgraph $G\setminus\{u,w\}$ is a tree on $n-2$ vertices, thus 
	$\Delta(G\setminus\{u,w\})$ has $n-3$ maximal cliques, which implies that $c(G)=n-2, $ contradiction.

\emph{Case 2.} There exists a triangle $F\in \Delta(G\setminus\{u,w\}).$ Then, as $G$ satisfies condition (iii), there exists a triangle 
$F^\prime\neq F_c$ which intersects $F_c$ along an edge. But this is impossible since the branch $F_j$ intersects $F_c$ in one vertex. 

Finally, we have to consider that $F_j$ shares an edge with $F_c,$ say $F_j\cap F_j=\{v,w\}.$ Since $F_j$ is a branch of $F_c,$ there is no other facet $F$ of $\Delta(G)$ with $F\cap F_c=\{u,w\}$ or $F\cap F_c=\{u,v\}.$ Then the graph $G\setminus u$ obviously satisfies conditions (i)--(iii) and has $n-1$ vertices. By the choice of $G,$ we have $c(G\setminus u)=n-3,$ thus $c(G)=n-2,$ contradiction. 
\end{proof}

\begin{Theorem}\label{thm:chordal}
Let $G$ be a connected chordal graph on the vertex set $[n].$ Then the following statements are equivalent:
\begin{itemize}
	\item [\emph{(i)}]  $(J_G)_{\mm}\subset R$ is licci.
	\item [\emph{(ii)}] $J_G$ is  Cohen-Macaulay and $n-2\leq \reg(S/J_G)\leq n-1.$
	\item [\emph{(iii)}] $J_G$ is  unmixed and $n-2\leq \reg(S/J_G)\leq n-1.$
	\item [\emph{(iv)}] $G$ is a path graph or it is isomorphic to a graph depicted in Figure~\ref{fig:licci}.
\end{itemize}
\end{Theorem}

\begin{proof}
We have to prove only the implication (iii) $\Rightarrow$ (iv). Let $J_G$ be unmixed and let $\reg(S/J_G)= n-1$. Then, by 
Theorem~\ref{regmax}, $G$ is a path graph. Let us now discuss the case when $\reg(S/J_G)= n-2.$ By \cite[Theorem 3.5]{RSK}, we have $\reg (S/J_G)\leq c(G).$ Thus, we get $c(G)\geq n-2.$ If $c(G)=n-1,$ then $G$ is a tree, but since $J_G$ is unmixed, by \cite[Corollary 1.2]{EHH}, it follows that $G$ is a path graph.

As in the proof of Theorem~\ref{thm:licci}, it is enough to show that there is no indecomposable chordal graph with $n\geq 4$ vertices which satisfies the conditions $J_G$  unmixed and  $\reg(S/J_G)=c(G)=n-2.$ Let us assume that such a  graph $G$ does exists. 

By Theorem~\ref{NaokiThm}, it follows that the maximal cliques of $G$ have at most three vertices.
As $G$ is a chordal graph, by Dirac's theorem, 
 it follows that the facets of the clique complex $\Delta(G)$ of $G$ have a leaf order, say
$F_1,\ldots,F_{n-2}.$ In particular, this means that $F_{n-2}$ has a branch. Let $F_j$ with $j\leq n-3$ be a branch of $F_{n-2}.$

\textbf{Case 1.} Assume that the intersection $F_j\cap F_{n-2}$ consists of only one vertex of $G,$ say $F_j\cap F_{n-2}=\{v\}.$ If $F_{n-2}$ has only the branch $F_j$, then $G$ is decomposable which contradicts our assumption on $G.$
Thus $F_{n-2}$ has $q\geq 2$ branches, say  $F_{j_1},\ldots,F_{j_q}.$ Then, as $J_G$ is unmixed, it follows that 
the induced subgraph of $G\setminus v$ on the vertex set $\bigcup_{i=1}^q F_{j_i}\setminus v$ is connected. This implies that all the facets $F_{j_1},\ldots,F_{j_q}$ are triangles.  If $F_{n-2}$ is also a triangle, we  get a contradiction to Lemma~\ref{codim1}. Thus, $F_{n-2}$ must be an edge and then $v$ is a cut vertex of $G$ with $\deg_G(v)\geq 4.$ By Lemma~\ref{lm:4neighbors}, it follows that 
$\reg(S/J_G)\leq n-3,$ a contradiction. \\
\textbf{Case 2.} Assume that the intersection $F_j\cap F_{n-2}$ consists of two vertices  of $G,$ say $F_j\cap F_{n-2}=\{v,w\}.$ In this case, $F_{n-2}$ is a triangle with the vertices $u,v,w.$
Since  $J_G$ is unmixed, there must be other facets of $\Delta(G)$ whose intersection with $F_{n-2}$ is contained in $\{v,w\}$ or equal to $\{v,w\}$. 
Let $F_{j_1},\ldots,F_{j_q}$ with $q\geq 2$ and $j_q=j$ be the facets of $\Delta(G)$ with $F_{j_s}\cap F_{n-2}\subseteq \{v,w\}$ for $1\leq s\leq q.$  As $v$ is not a simplicial vertex in $G,$ we may apply  again  \cite[Lemma 4.8]{Ohtani} and  get 
\[
J_G=J_{G_v} \cap (J_{G\setminus v}+(x_v,y_v)).
\]
We use   the following exact sequence of $S$--modules:
\[
0\to \frac{S}{J_G} \to \frac{S}{J_{G_v}}\bigoplus \frac{S}{J_{G\setminus v}+(x_v,y_v)}\to \frac{S}{J_{G_v\setminus v}+(x_v,y_v)}\to 0.
\] to derive that 
\begin{equation}\label{eq:final}
\reg\frac{S}{J_G}\leq \max\{\reg \frac{S}{J_{G_v}}, \reg \frac{S}{J_{G\setminus v}+(x_v,y_v)}, \reg \frac{S}{J_{G_v\setminus v}+(x_v,y_v)}+1\}.
\end{equation}

By \cite[Lemma 3.4]{RSK}, it follows that $c(G_v)\leq c(G)-q$, hence, by our assumption on $q,$ we get $c(G_v)\leq n-4.$ On the other hand, by \cite[Lemma 3.3]{RSK}, we have 
$c(G_v\setminus v)\leq c(G_v)$, thus $c(G_v\setminus v)\leq n-4.$ In particular, it follows that 
\begin{equation}\label{ineq}
\reg(S/J_{G_v})\leq n-4 \text{ and } \reg(S/J_{G_v\setminus v}+(x_v,y_v))\leq n-4.
\end{equation}
Therefore, by (\ref{eq:final}), we must have 
\[
\reg \frac{S}{J_{G\setminus v}+(x_v,y_v)}=\reg \frac{S'}{J_{G\setminus v}}\geq n-2.
\] where $S^\prime=K[\{x_j,y_j\}:j\in [n]\setminus\{v\}].$ As $G\setminus v$ has $n-1$ vertices, it follows by Theorem~\ref{regmax} that $G\setminus v$ is a path graph. But in this case, $\MS=\{v,w\}$ is a cut set of $G$ because $G$ is indecomposable. In addition,  the restriction of 
$G$ to the vertex set $[n]\setminus \{v,w\}$ has two connected components, which is a contradiction to the unmixedness of $J_G.$ 
\end{proof}


\begin{thebibliography}{}

\bibitem{BMS} D. Bolognini, A. Macchia, F. Strazzanti, \emph{Binomial edge ideals of bipartite graphs}, European J. Combin. \textbf{70} 
(2018), 1--25.


\bibitem{Dir} G. A. Dirac, \emph{On rigid circuit graphs}, Abh. Math. Semin. Univ. Hambg. \textbf{38} (1961),
71�6.

\bibitem{EHH} V. Ene, J. Herzog, T. Hibi, \emph{Cohen-Macaulay binomial edge ideals}, Nagoya Math. J. \textbf{204} (2011), 57--68.

\bibitem{HNTT} H. T. H\`a, H. D. Nguyen, N. V. Trung, T. N. Trung, \textit{Symbolic powers of sums of ideals},  Math. Z. 
https://doi.org/10.1007/s00209-019-02323-8

\bibitem{HH10}  J. Herzog, T. Hibi, \emph{Monomial Ideals}, Grad. Texts in Math. \textbf{260}, Springer, London, 2010.

\bibitem{HHHKR} J. Herzog, T. Hibi, F. Hreinsdotir, T. Kahle, J. Rauh, {\em Binomial edge ideals and conditional independence statements}, 
 Adv. Appl. Math. \textbf{45} (2010), 317--333.

\bibitem{HHO} J. Herzog, T. Hibi, H. Ohsugi, \emph{Binomial ideals}, Grad. Texts in Math. \textbf{279}, Springer, London, 2018.

\bibitem{HeRi} J. Herzog, G. Rinaldo, \textit{On the extremal Betti numbers of binomial edge ideals of block graphs}, Electron. J. Combin.  \textbf{25}(1) (2018), \#P1.63

\bibitem{Hu} C. Huneke,  \textit{Linkage and Koszul homology of ideals}, Amer. J. Math. {\bf 104} (1982), 1043--1062.  

\bibitem{HuUl}  C. Huneke,  B. Ulrich, \emph{The structure of linkage}, Ann. of Math. \textbf{126} (1987), 277--334.

\bibitem{Joh} M. R. Johnson, \emph{Linkage and sums of ideals}, Trans. Amer. Math. Soc. 350 (5) (1998), 1913--1930.

\bibitem{KiMa2} D. Kiani, S. Saeedi Madani, \emph{Some Cohen-Macaulay and unmixed binomial edge
ideals}, Comm. Algebra \textbf{43} (2015), 5434--5453.

\bibitem{KiMa} D. Kiani, S. Saeedi Madani, \textit{The Castelnuovo--Mumford regularity of binomial edge ideals}, J. Combin. Theory Serie A \textbf{139} (2016), 80--86.

\bibitem{KTY} K. Kimura, N. Terai, K. Yoshida, \emph{Licci squarefree monomial ideals generated in degree two or with deviation two}, J. Algebra \textbf{390} (2013), 264--289.


\bibitem{MaMu} K. Matsuda, S. Murai, \emph{Regularity bounds for binomial edge ideals}, J. Commut. Algebra 5(1), 141--149 (2013).

\bibitem{MoSh} F. Mohammadi, L. Sharifan, \emph{Hilbert function of binomial edge ideals}, Comm. Algebra \textbf{42} (2014), 688--703.

\bibitem{Ohtani} M. Ohtani, \emph{Graphs and ideals generated by some $2$--minors}, Comm. Algebra, \textbf{39} (2011), 905--917.

\bibitem{Peeva} I. Peeva, \emph{Graded syzygies}, Algebr. Appl. \textbf{14} Springer, 2010.

\bibitem{PeSz}  C. Peskine,  L. Szpiro, \emph{Liaison des vari\'et\'es alg\'ebriques}, I, Invent. Math. \textbf{26} (1974), 271--302. 

\bibitem{RaRi} A. Rauf, G. Rinaldo,  \emph{Construction of Cohen-Macaulay binomial edge ideals}, Comm. Algebra \textbf{42} (2014), 238--252.  

\bibitem{Rin} G. Rinaldo, \emph{Cohen-Macaulay binomial edge ideals of small deviation}, Bull. Math.
Soc. Sci. Math. Roumanie \textbf{56}(104)(4) (2013), 497--503.

\bibitem{Rin2} G. Rinaldo, \emph{Cohen-Macaulay binomial edge ideals of cactus graphs}, J. Algebra  Appl. \textbf{18} No. 4 (2019) 1950072, pp. 1-18.

\bibitem  {RSK} M. Rouzbahani Malayeri, S. Saeedi Madani, D. Kiani, \emph{Regularity of binomial edge ideals of chordal graphs}, arXiv:1810.03119 

\bibitem{Sara} S. Saeedi Madani, \textit{Binomial Edge Ideals: A Survey}, in Multigraded Algebras and Applications (V. Ene, E. Miller Eds.)
Springer Proceedings in Mathematics \& Statistics (2018), 83--94.




\end{thebibliography}
\end{document}